\documentclass[onefignum,onetabnum]{amsart}

\usepackage{lipsum}
\usepackage{amsfonts}
\usepackage{graphicx}
\usepackage{epstopdf}
\usepackage{algorithmic}
\usepackage{verbatim}
\usepackage{cleveref}
\usepackage[foot]{amsaddr} 

\theoremstyle{plain} 
\newtheorem{theorem}{Theorem}[section]
\newtheorem{lemma}[theorem]{Lemma}
\newtheorem{proposition}[theorem]{Proposition}

\theoremstyle{definition} 
\newtheorem{definition}[theorem]{Definition}
\newtheorem{example}[theorem]{Example}
\newtheorem{assumption}[theorem]{Assumption}

\theoremstyle{remark} 
\newtheorem{remark}[theorem]{Remark}

\crefname{hypothesis}{Hypothesis}{Hypotheses}

\title{Random walks on dense graphs and graphons}

\author{Julien Petit$^{1,2}$}\author{Renaud Lambiotte$^3$}\author{Timoteo Carletti$^2$}
\address{$^1$ Department of Mathematics, Royal Military Academy, Brussels, Belgium}
\address{$^2$ Department of Mathematics and  Namur Institute for Complex Systems (naXys),  University of Namur, Belgium}
\address{$^3$ Mathematical Institute, University of Oxford, UK}
\email{julien.petit@unamur.be}
\email{renaud.lambiotte@maths.ox.ac.uk}
\email{timoteo.carletti@unamur.be}

\usepackage{amsopn}

\DeclareMathOperator{\ran}{ran}
\DeclareMathOperator{\str}{str}

\DeclareMathOperator*{\esssup}{ess \, sup}

\DeclareMathOperator{\dom}{dom}
\DeclareMathOperator{\spanset}{span}

\usepackage{mathtools}
\usepackage{amssymb}
\DeclarePairedDelimiter\norm{\lVert}{\rVert} 
\DeclarePairedDelimiter\abs{\lvert}{\rvert}

\begin{document}

\maketitle

\begin{center}
	{\small
\today }
\end{center}

\begin{abstract}
Graph-limit theory focuses on the convergence of sequences of increasingly large
graphs, providing a framework  for the study of dynamical systems on massive graphs, where classical methods would become computationally intractable. Through an approximation procedure, the standard ordinary differential equations are replaced by nonlocal evolution equations on the unit interval. 
In this work, we adopt this methodology to prove the validity of the continuum limit of random walks, a  largely studied model for diffusion on graphs. 
We focus on two classes of processes on dense weighted graphs, in discrete and in continuous time, whose dynamics are encoded in the transition matrix of the associated Markov chain,  or in the random-walk Laplacian.
We further show that previous works on the discrete heat equation, associated to the combinatorial Laplacian, fall within the scope of our approach. Finally, we characterize the relaxation time of the process in the continuum limit.
\end{abstract}

\smallskip
\textbf{Key words.}
  random walk, dense graph, graphon, continuum limit

\textbf{AMS subject classification.}
	05C81, 34G10,  45K05, 47D06


\section{Introduction}
\label{sec:introduction}

Graphs are everywhere, they appear for example in physics, engineering, biology, climat science, expectedly every time discrete entities interact through links of some nature~\cite{barabasi2016network}. They offer a conceptually simple but general enough modeling approach to real-life networks of various degrees of complexity. But as very large graphs have become commonplace in scientific research and real-world applications --   online social networks or the brain are   relevant examples -- a range of graph-theoretical methods, algorithms and computational problems on graphs face scalability issues. In a recent series of works on the continuum limit of graphs, graphons have emerged as an appropriate limit object, defined when the number of nodes goes to infinity~\cite{lovasz2006limits,borgs2008convergent,borgs2011limitsrandomsequences,lovasz2012large}. They offer a both elegant and efficient workaround allowing for the analysis of massive graphs, while simultaneously providing a non-parametric network generation method that reaches far beyond the classical stochastic block model~\cite{holland1983stochastic}. The versatility of graphons and their practical interest is revealed in their present-day use in many diverse domains, e.g. network identification~\cite{gao2015rate}, power network dynamics~\cite{kuehn2019power}, epidemics spreading~\cite{vizute2020epidemics},  reaction-diffusion~\cite{kaliuzhnyi2017semilinear}, synchronization of oscillators~\cite{medvedev2014nonlinearDENSE, medvedev2018kuramoto} and game theory~\cite{parise2019graphon}. However, the majority of these examples focus on a specific type of diffusion, namely Fickean diffusion,   while alternative  random-walk processes are usually preferred to build network algorithms for node ranking~\cite{Brin1998conf} or community detection~\cite{Delvenne2010PNAS}. This constitutes an important limitation that we address in this paper.

Diffusion on finite graphs is an extensive topic of research, which is relevant both from a theoretical and an applied perspective, and is often modeled as some variant of random walk process. 
Random walks on graphs are useful in many ways. They can for example identify clusters of well-connected nodes, also known as communities~\cite{rosvall2009map, Delvenne2010PNAS}, or measure the relative importance or centrality of the nodes in a networked system~\cite{Brin1998conf, langville2004deeper,lambiotte2012ranking}, and as mentioned are   a paradigm for various diffusive and spreading processes on graphs~\cite{masuda2017random}.  
There are overall three dominant classes of random walks. The first one is the  discrete-time walk, in which case the walker performs a new jump at every discrete  time step. The destination node of the jump is uniformly chosen among the neighbors   in the graph structure. Secondly there is the continuous-time, node-centric variant. The difference resides in that the jumps take place at any point in time, as dictated by a continuous random variable governing the resting time of the walker on a node. Finally, the third class corresponds to the continuous-time edge-centric walk, which can  interestingly be viewed as the discrete version of the heat equation. These are two different forms of normal diffusion equations that can be derived from Fick's law, that essentially implies a flux from regions of high to low concentration, across a concentration gradient.  The difference between the node-centric and edge-centric processes  is clear when observing the matrices controlling their dynamics, the random-walk and the combinatorial Laplacian respectively~\cite{petit2019classes}.

In this work, we first revisit existing results for the continuum limit of the discrete heat equation  and some nonlinear variants. This limit was the subject of a series of  papers recently~\cite{medvedev2014nonlinearDENSE,medvedev2014small,medvedev2018kuramoto}, but a random walk interpretation is useful and was still lacking.  We then concentrate on the continuum limit of the node-centric case, hence considering the limit of the  random-walk Laplacian operator. In general, for non-regular graphs, this operator differs from the combinatorial Laplacian, which is often preferred in algorithmic implementations such as spectral clustering~\cite{von2008consistency}, because it properly accounts for the heterogeneous degree distributions observed in real-life networks.  The random-walk operator in this work shouldn't be confused with a another operator common in the machine learning community,  also called random-walk Laplacian,  
which has an established convergence to  the Laplace-Beltrami operator~\cite{hein2007graph,BELKIN20081289}. 

Our approach is based on graph-limit theory~\cite{lovasz2012large}, which does not rely on the assumption that the data generating the graphs is sampled from a distribution on a manifold~\cite{gine2006empirical,rosasco2010learning}.  Our main contribution is to prove the convergence of the space-discrete problem to a continuous problem in some appropriate setting. The problem on the  continuum then falls in the realm of nonlocal evolution equations, as it is a volume-constrained diffusion problem~\cite{d2017nonlocal}. Its analysis is  limited  to some   consequences of spectral theory applied to our operators. 

Importantly, graph-limit theory defines a framework for the convergence of graphs of increasing size, but it may as well be seen as   a possibly random graph-generating method. From that perspective, our work demonstrates that one may analyze the continuum model, to draw valid conclusions regarding the dynamics on the graphs generated by that model. As such, we answer positively the question of transferability, showing that we can avoid the cost of repeatedly storing each graph, designing, visualizing or analyzing it,  or computing the associated spectral properties, a computationally demanding task for large graphs~\cite{morgan1991computing}.

The paper is organized as follows. \Cref{sec:preliminaries} contains the basics about graph-limit theory, and introduces graphons as the limit objects of dense graph sequences. A short presentation of the main random walk models opens~\cref{sec:random-walk}. Then follows a random walk interpretation of the continuum limit of the heat equation on graphs, before we focus on our main concern, the continuous-time node-centric walk. Well-posedness of the continuum problem is the subject of \cref{sec:well-posedness}. The main convergence results are presented in~\cref{sec:Convergence-on-dense-weighted-graphs}. These results apply to dense graphs, and follow from a semigroup approach. We distinguish between different scenarios : first the discrete problem on graphs is sampled from the continuum version, and then the other way around. We then proceed with an analysis of the relaxation-time of the process based on spectral theory in~\cref{sec:relaxation}. In~\cref{sec:discrete-time} we revisit the discrete-time problem, before the conclusion of~\cref{sec:conclusions}.

\section{Preliminaries}
\label{sec:preliminaries}
We first set the notation and recall definitions for various graph-related concepts, following closely~\cite{woess2000random}. For the sake of self-consistency we then introduce key notions about graph-limit theory~\cite{lovasz2012large}. 

\subsection{Graphs}
\label{sec:graphs}
Let a $G=(V,E)$ be a graph where $V$ is  a finite   set of vertices (or nodes), and $E \subset V \times V$ is the set of edges. We consider a symmetric adjacency relation~$\sim$ such that the  graph is undirected. 
Each edge may be attributed a weight,  making the otherwise unweighted graph into a weighted one. Let $\vert V \vert$ and  $\vert E \vert $ respectively denote the number of vertices and edges.  The density $\rho $ of the graph is the fraction of edges that are actually present, compared with the maximum possible number of connections~: $\rho = \frac{\vert E \vert}{\binom{|V|}{2}}  $.  When it makes sense to take the limit $\vert V \vert  \rightarrow \infty$, one says the graph is dense if $\vert E \vert = O(\vert V\vert^2) $, and sparse otherwise. 

 The number of neighbors of node $v$  is denoted by $\deg(v)$. In weighted graphs,   $\str(v)$ stands for the weighted degree, or strength, of  node $v$, namely  the sum of the weights of all edges attached to $v$. For the sake of simplicity, the notation $\str(v)$ will henceforth also apply to unweighted graphs and will refer to the degree, thereby identifying unweighted graphs with graphs with binary weights, either~0 or~1. 

A  path  between two nodes $v, w$  is an ordered sequence of nodes $ \left[v_1,\ldots, v_n\right]$ such that $v = v_1$, $w=v_n$ and $v_i \sim v_{i+1}$, $i=1,\ldots, n-1$. 
A graph is connected if every pair of nodes is linked by a path. Let $\mathbb M_n$ be the space $n \times n$ matrices. The adjacency matrix $A \in \mathbb M_n$ of a   graph $G$ with $n$ vertices is the square matrix where $A_{ij}$ is the weight of the edge between nodes with labels $i$ and $j$ and zero if no such edge exists. Unweighted graphs have binary adjacency matrix.

\subsection{Graphons}
\label{sec:graphons}
Recent research~\cite{lovasz2006limits,borgs2008convergent,borgs2011limitsrandomsequences,lovasz2012large} provides a theoretical framework to study convergence of dense graphs sequences. As a starting point, the so-called cut (or rectangular) metric allows to define the notion of Cauchy sequence of graphs of increasing number of nodes. Their limit object, called graphon, is a symmetric Lebesgue-measurable function $W : [0,1]^2 \rightarrow [0,1]$\footnote{Note that this choice of domain and range is somehow restrictive by comparison with other works where graphons may be unbounded and only feature some integrability property.
However, we will work with the standard definition because it achieves the desired degree of generality.}. Therefore, the space of graphons is essentially the completion of the set of finite graphs seen as step functions (see Section~\ref{sec:from-graphons-to-graphs-and-back}), endowed with the so-called cut metric\footnote{%
	There is a different though equivalent notion of convergence for dense graph sequences. It is  called subgraph convergence, and is defined via associated sequences of induced subgraph densities~\cite{borgs2008convergent}. 
}
which we introduce hereafter in its graphon version~\cite{lovasz2012large, medvedev2014nonlinearDENSE, gao2017control}. 
The cut norm for graphons  is given by 
$$\Vert W \Vert_\Box = \sup_{S,T \subset \mathfrak M[0,1]} \iint_{S \times T}W(x,y) dx dy, $$
where the supremum is over measurable subsets of $[0,1]$.  The notation $\norm{W}_p$ refers to the usual $L^p$ norm of function defined on $[0,1]^2$, for $1\leq p \leq \infty$. 
The following inequalities are immediate consequences of this definition, and of the inclusion theorem of $L^p$ spaces on finite measure spaces~:
\begin{equation}\label{eq:norms}
\Vert W \Vert_\Box \leq \Vert W \Vert_1 \leq \Vert W \Vert_2 \leq  \Vert W \Vert_\infty \leq 1. 
\end{equation}
Graphons are unique up to a composition with an invertible measure preserving mapping $\phi  :[0,1] \rightarrow [0,1]$, which amounts to  invariance of the limit graphon with respect to a relabeling of the nodes of the graphs. The graphons  $W^\phi$ defined by $W^\phi(x,y) = W(\phi(x),\phi(y))$  and $W$ are in the same equivalence class. The cut metric $\delta_\Box$  between two graphons $U$ and $W$ is therefore defined by 
\begin{equation}
\delta_\Box(U,W) = \inf_{\phi \in \mathfrak L } \Vert U^\phi - W \Vert_\Box
\end{equation}
where $\mathfrak L$ is the space of the Lebesgue measurable bijections on the  unit interval. The definition is similar for the  $\delta_p(\cdot,\cdot)$ metrics based on the $L^p$  norms, $1\leq p \leq \infty$. Since two different graphons $U,W$ can satisfy $\delta_\Box (U,W) = 0$, strictly speaking $\delta_\Box$ is a metric only when we identify such graphons $U$ and $W$ \cite{borgs2008convergent}. Let  us denote by $\mathbb W$ the space of graphons after this identification. 

It holds that the metric space $(\mathbb W,\delta_\Box)$ is compact, namely sequences of graphons posses at least one convergent subsequence in the cut metric. Unless explicitly mentioned,  in this work we assume convergence of graphons in the $L^2$ norm topology. Hence by completeness,  Cauchy sequences in $(\mathbb W, \norm{\cdot, \cdot}_2)$ converge in the $L^2$ metric, and thus also in the $\delta_2$ and $\delta_\Box$ metrics, the limit being the same. 

Many attributes of graphs   have natural counterparts in the realm of graphons. A prominent example is the notion of strength, which plays a key role in this paper.  For a given graphon we let 
$
	k(x) : = \int_0^1 W(x,y) dy
$
denote the (generalized) degree function. Since in this work graphons are bounded function $W:[0,1]^2 \rightarrow [0,1]$, the degree function  is bounded, $0 \leq k(x) \leq 1$ for all $x \in [0,1]$. 
 
\subsection{Graphs as step graphons and graphs from graphon models}
\label{sec:from-graphons-to-graphs-and-back}
The connection between graphs and graphons is a two-way street. First, graphs can be mapped to the graphon space through a step function representation of their adjacency matrix. Let $\mathcal P = \left\{ P_1,\ldots, P_n  \right\}$ be a uniform partition of $[0,1]$, where $P_i = \left[ \frac{i-1}{n}, \frac{i}{n}\right)$ for $i = 1, \ldots, n-1$, and $P_n = \left[  \frac{n-1}{n}, 1  \right]$. Then let $\eta :  
\mathbb M_n \rightarrow \mathbb W$ be a mapping such that
\begin{equation}
\eta (G) (x,y) = \sum_{i=1}^{n} \sum_{j= 1}^{n} A_{ij} \chi_{P_i}(x) \chi_{P_j}(y), 
\end{equation}
where $\chi_S$ is the indicator function of set $S$ and $A$ the adjacency matrix of graph $G$. The mapping  thus defines the step (or empirical) graphon $\eta (G)$ associated to $G$. Similarly, $\eta$ maps vectors $\mathbf u = (u_1,\ldots, u_n)$ to piecewise constant functions on $[0,1]$, so that
\begin{equation}
\eta (\mathbf  u)(x) = \sum_{i =1}^n u_i \chi_{P_i}(x). 
\end{equation}

On the other hand, graphons can be considered as deterministic or  (exchangeable~\cite{diaconis2007graph}) random graph models, but in this work we adopt and present the   deterministic setting. Let $W \in \mathbb W$ be a graphon and let the integer $n$   denote the desired   number of nodes in the graph. Then $W$ generates a dense graph by  assigning weights  to the edges, which can be done in two ways.  In a first approach, the weight $A_{ij}$ of the edge between two nodes~$i$ and  $j$ equals the  mean value of $W$ on the corresponding cell of the partition of the unit square: 
\begin{equation}\label{eq:quotient-graph}
A_{ij} = n^2 \int_{P_i} \int_{P_j} W(x,y) dx dy,       \quad i,j = 1,\ldots, n. 
\end{equation}
This results in the so-called quotient graph $W/\mathcal P$. One can prove that there is almost everywhere point-wise convergence of the associated step graphon $\eta (W/\mathcal P)$ to $W$ (\cite{borgs2008convergent}, lemma 3.2). 

A second approach to generate a graph from a given graphon $W \in \mathbb W$, is to define 
\begin{equation}\label{eq:sampled-graph}
A_{ij} = W\left (\frac{i}{n}, \frac j n \right) ,       \quad i,j = 1,\ldots, n, 
\end{equation}
in a way that is reminiscent of $W$-random graphs \cite{lovasz2006limits}.  Let us denote  $W_{[n]}$  the corresponding graph. Observe that $\eta \left(W_{[n]} \right) \rightarrow W$ point-wise at every point of continuity of $W$~\cite{medvedev2014nonlinearDENSE}. 

\subsection{Graphons as kernels of operators}
\label{eq:kernels}

Every graphon $W \in \mathbb W$ can be considered as a kernel, allowing to formally define an integral operator $\mathcal W$ on functional  spaces on $[0,1]$ through
\begin{equation}\label{eq:HSoperator}
\mathcal W f (x) = \int_0^1 W(x,y) f(y) dy. 
\end{equation}
The composition (product) of two such operators is given by 
\begin{equation}
\mathcal U \mathcal W f(x) = \int_0^1 (U\circ W)(x,y) f(y) dy, 
\end{equation}
where $\circ$ is the operator product between the graphon kernels, defined by 
\begin{equation}
\label{eq:UcircW}
(U \circ W) (x,y) = \int_0^1 U(x,z) W(z,y) dz, \quad \forall x,y \in [0,1]. 
\end{equation}
Observe that in general, $U \circ W$ is not a symmetric function. We denote $W^{\circ n}$ the operator product of  the kernel, as opposed to the point-wise product $W^n(x,y) = (W(x,y))^n$, which is associated to   the operator $\mathcal W^n$. It follows from \cref{eq:UcircW} that
$
W^{\circ n}(x,y) = \int W(x,z_1) W(z_1,z_2)   \ldots   W(z_{n-1},y) dz_1 dz_2 \ldots dz_{n-1}. 
$

\section{Random walks and their continuum limit}
\label{sec:random-walk}
The aim of this section is twofold. Firstly, we introduce the three main random walk models.  Secondly, we give a random walk perspective on  the continuum limit of the discrete heat equation, namely the edge-centric walk. We then formally derive the continuum limit of the so-called node-centric   walk. 
Our approach implies edges are directed, because they are associated with possible moves of the walker, with an origin and a destination. Therefore, the symmetry of the adjacency matrix indicates there exists a reciprocal  to each edge and that both have the same weight. Further, we may assume that the graph is connected, otherwise the random walk  is considered independently on each connected component, that is, each connected subgraph that is connected to no other additional node of the original graph.

\subsection{Random walks in discrete  and continuous time}

In discrete-time, we introduce  a random walk on a  connected graph as  a Markov chain where $V$ is the state-space and the transition probability  from node $v_i$ to $v_j$ is encoded in the matrix 
\begin{equation}
T_{ij}= \begin{cases}
1/\str \left(v_i\right) & \mbox{ if } v_i \sim v_j, \\
0             & \mbox{ otherwise.}
\end{cases}
\end{equation}
Let $\mathbf p (\ell)= (p_1(\ell),\ldots, p_n(\ell))$ be the row vector of residence probabilities on the nodes, that is, $p_i(\ell)$ is the probability that the walker is located on node number $i$ after $\ell$ steps. Then
\begin{equation}\label{eq:DTRW}
\mathbf{p}(\ell+1) = \mathbf{p}(\ell) T ,
\end{equation} 
where  $T = D^{-1}A$. Here $D $ denotes the diagonal matrix of the strengths, or degrees in unweighted graphs.   It follows from \cref{eq:DTRW} that for any $\ell \in \mathbb{N}$,  $\mathbf{p}(\ell) = \mathbf{p}(0) T^\ell$.   

In the continuous-time node-centric variant, when the walker arrives on node $v$, a probability density function $\psi_v(t)$ determines the waiting-time until the next jump, in which case a destination node is selected uniformly among the neighboring ones. We limit ourselves to Poissonian walks,  for which the waiting-time follows a memoryless exponential distribution $\psi_v(t) = \mu_v \exp(-\mu_v t)$ with rate $\mu_v$ ($t\geq 0$). The master equation for  $u_i(t)$, the probability  to find the walker on node $i$ at time $t$,   reads
\begin{equation} \label{eq:angstmann}
\dot u_i(t) = \sum_{j = 1}^{N} \mu_j u_j \frac{1}{\str(v_j)} A_{ji} - \mu_i u_i, \qquad   i=1,\dots,n . 
\end{equation}
Assume that in~\cref{eq:angstmann} the rate $\mu_j$  on the nodes   is the same for all nodes, $\mu_j = \kappa >0$ for all $j$. Then $\kappa$ sets the timescale, and after a scaling of time, $t \mapsto \kappa t$, under matrix form the master equation~\cref{eq:angstmann} rewrites
\begin{equation}\label{eq:node-centric}
\dot{ \mathbf u} = \mathbf u (D^{-1} A-I), 
\end{equation}
where   $ \mathbf u(t) = (u_1(t),\ldots, u_n(t))$ is a row vector.  The matrix $L^{rw} = D^{-1} A - I$ is the random walk Laplacian. 
, 
Moreover, it is easy to show that the discrete-time walk and the continuous-time version  share the same  asymptotic state, and that it is proportional to $\left(\str(v_1), \ldots, \str(v_n)\right)$.

In the edge-centric variant\footnote{
The fact that this walk can be formulated in terms of edges dynamics, where the walker passively follows the  activations of the edges explains the alternative designation of ``fluid model''~\cite{masuda2017random}. It is the graph version of the heat equation on a continuum. 
},  the rate of the exponential distribution is proportional to the degree of the node, $\mu_j = \kappa \str (v_j)$, allowing a constant rate of jump across all edges of the graph. Hence in matrix form, \Cref{eq:angstmann} rewrites
\begin{equation} \label{eq:edge-centric}
\dot{ \mathbf u} = \kappa \mathbf u (A-D)
\end{equation}
Here, $L = A-D$ is called the combinatorial Laplacian of the graph. This model exhibits a homogeneous asymptotic state. 
Observe that the number of jumps is not trajectory-independent, as is the case in both the discrete- and continuous-time node-centric  walks.

\subsection{Formal derivation of the  continuum limit}
\label{sec:continuum-limit-node-centric}
Let us first  take a closer look at the edge-centric walk, and assume for simplicity an unweighted graph.  If $\kappa >0$, then $\kappa \deg(v_j) \rightarrow \infty $ if $\deg(v_j) \rightarrow \infty$, which will happen for some if not all nodes of a dense  graph. The walker would perform jumps at an infinite rate, which is physically unrealistic. Normalizing the rate of the process according to the number of vertices avoids this situation. If~$\kappa$ becomes dependent on $n$, say $\kappa_n = \frac{1}{n}$, the resulting rate in each node remains bounded, $\kappa_n \deg(v_j) \leq 1$ for all $j$ independently of the number of nodes.  This explains the normalization that was required to justify  the continuum limit of \cref{eq:edge-centric}  in~\cite{medvedev2014nonlinearDENSE}.

In contrast with the edge-centric model,  no normalization of the rate parameter~$\kappa$ of the node-centric walk is needed when the number of nodes grows to infinity, since the rate does not depend on the   structure of graph.  The continuum limit therefore directly applies to the unmodified discrete model. For a formal derivation in this case, consider again the  vector $\mathbf u (t) $ satisfying~\cref{eq:node-centric} and the uniform partition $\mathcal P = \left\{ P_1,\ldots, P_n  \right\}$ of $[0,1]$, with $u(\cdot, t) := \eta (\mathbf u (t))$ an associated step function on the interval.  
Let  $k_\eta$ denote the generalized degree function of the step graphon $\eta(G)$. 
Observe that this degree function  is actually the normalized strength (or also degree, when the graph is unweighted) of the nodes in $G$: 
\begin{align}
\str(v_i) & = n \sum_{j = 1}^n \int_{P_j} A_{ij} dy = n \sum_{j = 1}^n \int_{P_j} \eta(G)(x,y)dy = n k_\eta(x)
\end{align}
for all $x \in P_i$. It follows that 
\begin{align}
\sum_{j = 1}^n \frac{A_{ij}}{\str\left(v_j\right)} u_j(t) & = \sum_{j = 1}^n n \int_{P_j} \frac{  A_{ij}  }{   \str \left(v_j\right)   }u(y,t) dy \nonumber \\
																				 & = \sum_{j = 1}^n n \int_{P_j} \frac{  A_{ij}  }{  n k_\eta(y) }u(y,t) dy= \int_0^1 \frac{  \eta(G)(x,y)   }{ k_\eta(y)   }u(y,t) dy, 
\end{align}
for every $x \in P_i$. Hence, the node-centric walk  on the graph has an equivalent continuum domain formulation
\begin{equation} \label{eq:node-centric-step}
\frac{\partial }{\partial t } u(x,t)  = \int_0^1 \frac{\eta(G)(x,y)}{k_\eta(y)} u(x,t)dy - u(x,t). 
\end{equation}
The goal of this work is to prove convergence in the appropriate norm of the solution of \cref{eq:node-centric-step} to the solution of the  evolution equation on the continuum
\begin{equation} \label{eq:node-centric-graphon}
\frac{\partial }{\partial t } w(x,t)  = \int_0^1 \frac{W(x,y)}{k(y)} w(y,t) dy - w(x,t), 
\end{equation}
where $W$ is the limit graphon of $\eta(G)$ in the $L^2$ metric.

Observe that similarly, a discrete equation of the form~\cref{eq:node-centric} is obtained starting from~\cref{eq:node-centric-graphon}, when the  graph is $W/\mathcal P$ or $W_{[n]}$.

\section{Well-posedness of the continuum initial value problem}
\label{sec:well-posedness}
Before we   prove the above-mentioned convergence, let us determine whether  the up-to-now formal \cref{eq:node-centric-graphon}, together with initial condition $w(x,0) = g(x)$, defines   a well-posed initial-value problem~(IVP). 

\subsection{Connectedness of the graphon}
Care will be taken first regarding how connectedness in the graph translates to graphons, and how it affects the integrability of $W/k$. The following definition  follows from~\cite{janson2008connectedness,lovasz2012large}. 
\begin{definition}
	A graphon $W$ is connected if 
	$
	\int_{S \times \left( [0,1] \setminus S \right)}W(x,y) dx dy >0
	$ 
	for every $S \in \mathfrak M [0,1]$ with lebesgue measure $\mu (S) \in (0,1)$. 
\end{definition}

Notice at this point that the connectedness (or lack thereof) of the graphs $G_n$ of the sequence does not imply that of their limit~\cite{janson2008connectedness}. Indeed, one could always make all the (otherwise disconnected) graphs of the sequence connected by a adding each time a node connected to all other nodes. This would leave the limit unchanged. And conversely, disconnecting one node in each connected graph of the sequence would not change the limit either. Also note that if a graphon $W$ is (dis)connected, then so are all the kernels in the same equivalence class~(\cite{janson2008connectedness}, theorem 1.16). Let us now look into the implications of connectedness of the graphon on the positiveness of the degree function and hence on the definition of the random walk Laplacian operator. 
\begin{proposition}\label{prop:connected-graphon}
	Let $W$ be a connected graphon, then $k > 0$ $\mu$-almost everywhere (a.e.). 
\end{proposition}
\begin{proof}
Let $N_x = \left\{  y \in [0,1] \, : \, W(x,y)>0 \right\}$ denote the neighborhood of $x\in [0,1]$ in $W$. 
Since $W$ is connected,   $\mu(N_x) >0$ for $\mu$-almost every $x$~(\cite{janson2008connectedness},~lemma 5.1) and therefore, 
\begin{equation}
k(x) = \int_{N_x} W(x,y) dy > 0\ \mbox{ for } \ \mu\mbox{-a.e. } x. 
\end{equation} 
\end{proof}
\begin{remark}
The connectedness of the graphon does not imply however that the degree function is bounded away from zero, namely that there exists a constant $c$ such that $0 < c \leq k$ on $[0,1]$. Take for instance $W(x,y) = x^m y^m$ with $m >0$, for which $k(x) = x^m /(m+1)$.
\end{remark}
That $k$ can be arbitrarily small influences the integrability of the kernel $K(x,y):=\frac{W(x,y)}{k(y)}$ in~\cref{eq:node-centric-graphon}, as discussed in the following remark. 
\begin{remark}
	The connectedness of the graphon does not imply that the integral kernel $K(x,y)$ is in $L^p[0,1]^2$ for $p>1$.   Consider for example the binary graphon  $W= \chi_{x^\alpha+y^\alpha \leq 1} $ for $\alpha >0$,  where the subscript $x^\alpha+y^\alpha \leq 1$ is short for the set of couples $(x,y) \in [0,1]^2$ such that the inequality is satisfied. By a direct integration, $k(x) = \left(1-x^\alpha\right)^\frac{1}{\alpha}$.  The integral
	\begin{equation}\label{eq:counter-exemple}
	\norm{K}_p^p = \iint_{x^\alpha + y^\alpha \leq 1}(1-y^\alpha)^{-\frac{p}{\alpha}} dx dy = \int_0^1 (1-y^\alpha)^{\frac{1-p}{\alpha}} dy
	\end{equation}
	is finite if and only if $p<1+\alpha$. Hence, $K$ is in $L^2[0,1]$ only if $\alpha >1$, and in particular, the kernel $K$ of  the threshold graphon~\cite{diaconis2008threshold} obtained with $\alpha = 1$ is not square-integrable. 
	However, using Fubini-Tonelli it is easy to show that $\norm*{K}_1 = 1$  for all connected graphons, such that $K$ is always in $L^1[0,1]^2$. 
\end{remark}

Based on the preceding remark, in order to ensure that the kernel is square integrable, we will make the following assumption : 
	\begin{assumption}\label{hyp:bounded-away-from-0}
	There exists a constant $c$ such that  $0<c \leq k$ on $[0,1]$. 
\end{assumption}
If $W$ is bounded away from zero, so is $k$, but graphons with localized support may still fulfill the assumption, as shown by \Cref{fig:stripe-graphon}.

\begin{figure}[tbhp]
\centering
\includegraphics[width = .8\textwidth]{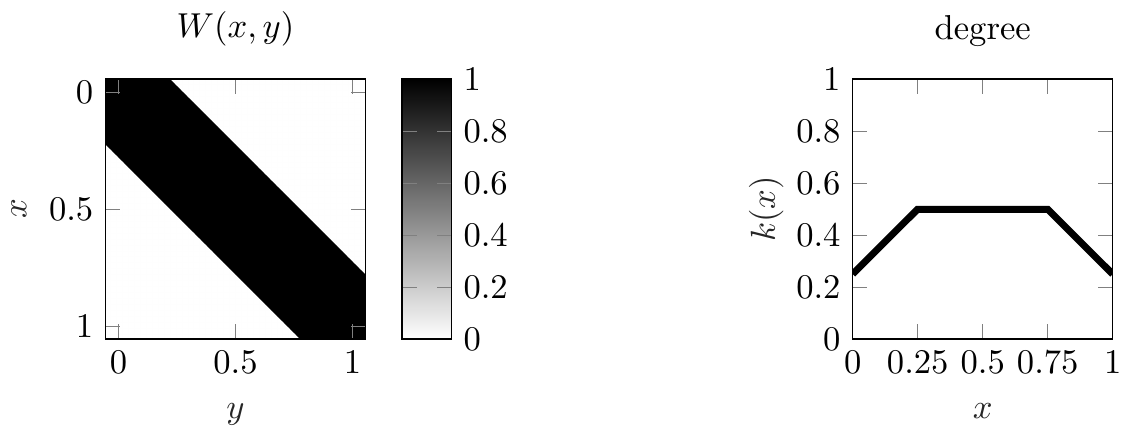}
\caption{The stripe graphon $W(x,y) = \chi_{|x-y| \leq 1/4}(x,y)$ (left panel) and its degree function (right panel). Observe that the support of $W$ is localized on a subset of the square, but $k$ is bounded away from zero.   }
\label{fig:stripe-graphon}
\end{figure}
\subsection{The IVP with functions in $L^2[0,1]$}

Resting on the operator in the right-hand side of~\cref{eq:node-centric-graphon}, we come to the following definition. 
\begin{definition}
	\label{def:RWLaplacian} Let $W \in \mathbb W$ be a connected graphon that verifies~\cref{hyp:bounded-away-from-0}. The random-walk Laplacian operator $\mathcal{L}^{rw} : L^2[0,1] \rightarrow L^2[0,1]$   is defined by
	\begin{equation}
\mathcal{L}^{rw} f(x) = \int_0^1 \frac{W(x,y)}{k(y)} f(y) dy -f(x). 
	\end{equation}
\end{definition}
By definition of $W$ and \cref{hyp:bounded-away-from-0},   $K(x,y) = W(x,y)/k(y)$ is a Hilbert-Schmidt kernel and $\mathcal K : L^2[0,1] \rightarrow L^2[0,1]$ defined by
\begin{equation}\label{eq:compact-op-markov-kernel}
\mathcal Kf(x)  =  \int_0^1 \frac{W(x,y)}{k(y)} f(y)dy, \quad \forall x \in [0,1] \mbox{ and } f\in L^2[0,1]
\end{equation}
is a compact Hilbert-Schmidt operator. Following~\cref{def:RWLaplacian}, the continuum IVP has the form
\begin{subequations}\label{eq:IVP}
\begin{align}
\frac{\partial }{\partial t}w(x,t) &= \mathcal L^{rw} w(x,t),\\
w(x,0) &= g(x) \in L^2[0,1]. 
\end{align}
\end{subequations}

\begin{theorem}
	\label{thm:existence-unicity-sol}
Let $W \in \mathbb W$ be connected and  satisfying~\cref{hyp:bounded-away-from-0}. Then there exists a unique classical solution to the initial-value problem~\cref{eq:IVP}. 
\end{theorem}
 
\begin{proof}
The operator $\mathcal K$ is linear, and continuous hence bounded. It follows that $\mathcal L^{rw}$ is linear and bounded. Hence it is closed. 
Therefore, $\mathcal L^{rw}$ is the infinitesimal generator of the (uniformly and thus) strongly continuous semigroup
\begin{equation}
\mathcal T^{rw}(t) = e^{\mathcal L^{rw} t} := \sum_{\ell=0}^\infty \frac{t^\ell \left(  \mathcal L^{rw}     \right)^\ell}{\ell !}. 
\end{equation}
Proposition 6.2 in~\cite{engel2001one} allows to conclude. 
\end{proof}
\begin{remark}
\label{rmk:classical-solution}
(Classical solution) By definition of  classical solution of the abstract Cauchy problem~\cref{eq:IVP}, the orbit maps
$
	t\in \mathbb R^{+} \mapsto w(x,t) \in L^2[0,1]
$ 
are continuously differentiable. 
\end{remark}

\begin{remark}\label{rmk:steady-state}
The asymptotic steady state $w_\infty $ of~\cref{eq:IVP} follows from $\mathcal L^{rw} w_\infty = 0$  and is proportional to the degree, $w_\infty \propto k$.  
\end{remark}
\subsection{Positivity} \label{sec:positivity}
The continuum IVP \cref{eq:IVP} would loose physical relevance if its solution were to loose the   positivity of the initial condition, $w(\cdot,0) \geq 0$. Before we proceed to a proof of positivity, let us first introduce a notation. For $g \in L^\infty [0,1]$, and $1 \leq p \leq \infty$,  let  $\mathcal M_g: L^p[0,1] \rightarrow L^p[0,1]$ denote the multiplication operator defined by 
$ 
\mathcal M_g f(x) = g(x) f(x). 
$

\begin{proposition}
\label{prop:positivity-of-sol}
Let $W$ be a connected graphon satisfying \cref{hyp:bounded-away-from-0} and let $w(\cdot,0) = g \geq  0$ be  the initial condition of IVP \cref{eq:IVP}. Then the classical solution $w(x,t)$ of the IVP satisfies $w(\cdot,t) \geq 0$ for all $t\geq 0$. 
\end{proposition}

\begin{proof}
Let us define $\widetilde{\mathcal L   } = \mathcal{M}_{1/k}    \mathcal L^{rw} \mathcal M_{k}$, yielding by a direct calculation
\begin{equation}\label{eq:new-L}
\widetilde{\mathcal L } f(x,t) = \frac{1}{k(x)} \int_0^1 W(x,y) f(y,t) dy-  f(x,t), \quad \forall f \in L^2[0,1].  
\end{equation}
Further let 
$u = \mathcal M_{1/k} w$ with $w(x,t)$ the solution of~\cref{eq:IVP}
such that 
$$\frac{\partial}{\partial t} u 
= \mathcal M_{1/k}  \frac{\partial}{\partial t}  w
= \mathcal M_{1/k}  \mathcal L^{rw} w 
= \mathcal M_{1/k}  \mathcal L^{rw} \mathcal M_k u 
= \widetilde{\mathcal L } u.$$
Since $w(\cdot,t) \geq 0 \iff u(\cdot,t) \geq 0$, it remains to prove the positivity of $u(\cdot, t)$.  Choose $\epsilon >0$ arbitrarily and let $v(x,t) = u(x,t) + \epsilon t$. Observe that $\widetilde{\mathcal L   } v = \widetilde{\mathcal L } u$,  and hence
$$
 \frac{\partial}{\partial t}  v - \widetilde{\mathcal L   } v
 =  \frac{\partial}{\partial t}  u + \epsilon - \widetilde{\mathcal L   } u 
 = \epsilon. 
$$
Let us show $v(x,t)$ reaches its minimum at some $(a,0)$, $a \in [0,1]$. Assume by contradiction that there exists $(a,\tau) \in [0,1] \times (0,T)$ for some $T >0$ such that $v(x,t) \geq v(a,\tau)$ for all $x$ and $t$. It follows that 
\begin{align*}
\widetilde{\mathcal L   } v(a,\tau) & = \frac{1}{k(a)} \int_0^1W(a,y) v(y,\tau) dy - v(a,\tau) \geq  \frac{v(a,\tau)}{k(a)} \int_0^1W(a,y)  dy - v(a,\tau) = 0. \nonumber \\
\end{align*}
Hence,  $\frac{\partial}{\partial t} v(a,\tau) = \widetilde{\mathcal L   } v(a,\tau)  + \epsilon = \epsilon >0$ which is in contradiction with the assumption of $v$ attaining its minimum in $(a,\tau)$ with $\tau >0$, so $\tau  = 0$. 
We have thus proved $v(x,t) \geq v(a,0)$, so that 
$$
u(x,t) + \epsilon t = v(x,t) \geq v(a,0) = u(a,0) = \frac{g(a)}{k(a)} \geq 0. 
$$
Since $\epsilon$ is arbitrary, this allows to conclude. 
\end{proof}

\subsection{The IVP with probability density functions}
Let us observe  that when $w(\cdot,t)$ in~\cref{eq:IVP} is a probability density function,   it is natural to consider $w(\cdot,t) \in L^1[0,1]$, and one may define $\mathcal L^{rw}$ as a mapping $L^1[0,1] \rightarrow L^1[0,1]$. Indeed, as in \cref{eq:compact-op-markov-kernel} let us still write  $\mathcal K $ the integral part of $\mathcal L^{rw}$ defined on $L^1[0,1]$.  By Fubini-Tonelli, the operator norm $\Vert \mathcal K \Vert_{1,1} := \norm*{\mathcal K}_{L^1[0,1] \rightarrow L^1[0,1]}$ satisfies
\begin{align}
 \Vert \mathcal K \Vert_{1,1} &  \leq \sup_{||f||_1 =1} \int_{[0,1]^2} \abs*{K(x,y)f(y)} dx dy = \sup_{||f||_1 =1} \int_{[0,1]} \abs*{f(y)} dy = 1. 	
\end{align}
This, combined with the fact that $\norm{\mathcal K f}_1 = 1$ if $f = 1$,  shows that ${\norm*{\mathcal K}_{1,1}  = 1}$, and so even without~\cref{hyp:bounded-away-from-0}, $\mathcal L^{rw}$ is a bounded mapping of $L^1[0,1]$ into itself. Additionally,   \cref{thm:existence-unicity-sol} about the existence and unicity of a solution to the IVP  has a similar formulation and proof in the present case. 
Further,  the positivity  established in \cref{sec:positivity}  also applies here, and this would still  not require~\cref{hyp:bounded-away-from-0}. The only significant change in the proof of~\cref{prop:positivity-of-sol} would be to use the   auxiliary operator $ \mathcal L^{rw} \mathcal M_{k}$ instead of $\widetilde{\mathcal L   } = \mathcal{M}_{1/k}    \mathcal L^{rw} \mathcal M_{k}$. When ${w(\cdot, 0) \geq 0}$  we further have conservation of the $L^1$ norm : 
\begin{equation}\label{eq:density-conservation}
\frac{\partial}{\partial t} \norm{w(\cdot,t)}_1 = \frac{\partial}{\partial t} \int_0^1 \abs{w(x,t)} dx  =  \frac{\partial}{\partial t} \int_0^1 {w(x,t)} dx  = 0. 
\end{equation}

In the remainder of the paper, for the sake of simplicity and in order to benefit from the Hilbert space framework at a later stage, we will however assume that  $W$ satisfies~\cref{hyp:bounded-away-from-0}. This allows to define  $\mathcal L^{rw}$ as an operator acting on   $L^2[0,1]$ and we do not use $L^1$ but  rather the stronger $L^2$ norms also present in  other works about dynamics on graphons~\cite{medvedev2014nonlinearDENSE, medvedev2018kuramoto}. 

\section{Convergence on dense   graphs}
\label{sec:Convergence-on-dense-weighted-graphs}

This section is divided in three parts. The first two parts show that the solution of the discretized problem on $W/\mathcal P$ or $W_{[n]} $  converges to that of the continuum IVP in norm $\norm{\cdot}_{C([0,T],L^2[0,1])}$ for any $T>0$. 
The goal of the  third part is to prove that the discrete problem can be approximated by its continuum version.

\subsection{Convergence on the quotient graph $W/ \mathcal P$}
\label{sec:convergence-on-quotient-graph}
Let us start with two simple lemmas. 
\begin{lemma} \label{lemma:step-kernel-on-step-graphon}Let $\mathcal A_\eta : L^2[0,1] \rightarrow L^2[0,1]$ be an integral operator with bounded kernel  $A_\eta$. Assume that $A_\eta$ is a.e.-constant on every cell $P_i \times P_j$ of the uniform partition of $[0,1]^2$. Further let $f \in L^2[0,1]$ and define $f_\eta$ by 
	\begin{displaymath}
	f_\eta(x) = n \sum_{i = 1}^n \int_{P_i} f(y) dy \chi_{P_i} (x), \quad \forall x \in [0,1]. 
	\end{displaymath}
	Then for all $\ell \in \mathbb{N}_0$, it holds that
	$
	\mathcal A_\eta^{\ell} f = \mathcal A_\eta^\ell f_\eta. 
	$
\end{lemma}

\begin{proof}
The proof  in the case $\ell = 1$ follows from a direct calculation, see for instance~\cite{gao2017control}, lemma~3. The claim for $\ell>1$ is   a direct consequence since then
\begin{displaymath}
\mathcal A_\eta^{\ell}f =\mathcal A_\eta^{\ell-1}\mathcal A_\eta f =  \mathcal A_\eta^{\ell-1}\mathcal A_\eta f_\eta = \mathcal A_\eta^{\ell}f_\eta. 
\end{displaymath}
\end{proof}

\begin{lemma}\label{lemma:difference-HS-operators} Let $\mathcal A, \mathcal B : L^2[0,1] \rightarrow L^2[0,1]$  be two Hilbert-Schmidt integral operators with respective  kernels $A$ and $B$ defined on the unit square, with $A \leq \beta$ for some constant $\beta >0$. Then, for all $f \in L^2[0,1]$ and $\ell \in \mathbb N_0$
	\begin{displaymath}
	\norm{\mathcal A^\ell f - \mathcal B^\ell f}_2 \leq \beta^{\ell-1} \norm{A-B}_2 \norm{f}_2 
	+ \norm{(\mathcal A^{\ell-1} - \mathcal B^{\ell -1}) \mathcal B f}_2 . 
	\end{displaymath}
\end{lemma}
 
\begin{proof}
Using the Minkowski inequality, we have 
\begin{align} \label{eq:Al-1-A-B}
	\norm{\mathcal A^\ell f - \mathcal B^\ell f}_2 & = \norm{\mathcal A^{\ell -1} \mathcal A f - \mathcal B^{\ell-1} \mathcal B f }_2    \nonumber\\
	&=  \norm{\mathcal A^{\ell -1} \mathcal A f  - \mathcal A^{\ell-1} \mathcal B f + \mathcal A^{\ell-1} \mathcal B f - \mathcal B^{\ell-1} \mathcal B f }_2  \nonumber \\
	&=  \norm{\mathcal A^{\ell -1} (\mathcal A f - \mathcal B f)    +( \mathcal A^{\ell-1} - \mathcal B^{\ell-1}) \mathcal B f }_2  \nonumber \\
	& \leq  \norm{\mathcal A^{\ell -1} (\mathcal A f - \mathcal B f) }_2  
	+\norm{( \mathcal A^{\ell-1} - \mathcal B^{\ell-1}) \mathcal B f }_2. 
\end{align}	
Now $\mathcal A^{\ell-1}$ is a Hilbert-Schmidt integral operator with kernel     $A^{\circ (\ell - 1)}$. For such operator, as a product of the Cauchy-Schwarz inequality it is known about the operator norm $\norm{\cdot} $   that 
$
\norm{\mathcal A^{\ell-1}} \leq \norm{A^{\circ (\ell - 1)}}_2
$
, or equivalently
\begin{equation}\label{eq:normHSoperator-norm2}
\norm{\mathcal A^{\ell-1} f}_2 \leq \norm{A^{\circ (\ell - 1)}}_2  \norm{f}_2. 
\end{equation}
The first term in the right hand side of~\cref{eq:Al-1-A-B} therefore satisfies
\begin{equation}\label{eq:A-B-f}
\norm{\mathcal A^{\ell -1} (\mathcal A f - \mathcal B f) }_2   
\leq \norm{A^{\circ(\ell -1)}}_2 \norm{\mathcal A f- \mathcal B f}_2 
\leq \beta^{\ell-1}  \norm{\mathcal A f- \mathcal B f}_2
\end{equation}
where we use  $\norm{A^{\circ(\ell -1)}}_2 \leq \norm{A}_2^{\ell-1}$ (\cite{gao2017control}, lemma 6)  and  $A(x,y) \leq \beta$ for all $0\leq x, y \leq 1$ to obtain the last inequality. Using again \cref{eq:normHSoperator-norm2} with $\ell = 2$, we also have
$
\norm{\mathcal A f-\mathcal B f}_2 \leq \norm{A-B}_2  \norm{f}_2
$
which, together with \cref{eq:Al-1-A-B} and \cref{eq:A-B-f} leads to the conclusion.  
\end{proof}

Now we are in a place to formulate the convergence  results. The continuous formulation of the discrete problem associated to \cref{eq:IVP} on the quotient graph reads\footnote{The subscript  $\square$   refers  to fact that the averaging is performed on square cells of $[0,1]^2$. To lighten the notations, we do not refer explicitly to the number of nodes of the graph, so we write $u(x,t)$ instead of, for instance,  $u^{(n)}(x,t)$.  }
\begin{subequations}\label{eq:IVP-quotient-graph}
	\begin{align}
	\frac{\partial }{\partial t}u(x,t) &= \mathcal L_\square ^{rw} u(x,t)\\
	u(x,0) &= g_\square(x)  
	\end{align}
\end{subequations}
where the random walk  Laplacian operator on $W/\mathcal P$ satisfies
\begin{equation}
\mathcal L_{\square} ^{rw} f(x)  = \int_0^1 \frac{\eta \left(   W / \mathcal P   \right)(x,y)}{k_\square (y)} f(y) dy - f(x), \quad \forall f \in L^2[0,1], 
\end{equation}
and the initial condition is averaged on each cell of the partition as 
\begin{equation}\label{eq:initial-cond-on-quotient-graph}
g_\square(x) = n\sum_{i = 1}^n \int_{P_i} g(y) dy \chi_{P_i}(x), \quad \forall x \in [0,1]. 
\end{equation}
Based on  the following proposition, operator $\mathcal L^{rw}_\square$ is well-defined. 
\begin{proposition}\label{prop:deg-of-connected-graphon}
Let $W$ be a connected graphon satisfying \cref{hyp:bounded-away-from-0}, then the strength of every node of the quotient graph determined by the partition $\mathcal P = \left\{ P_1,\ldots, P_n  \right\}$ of $[0,1]$  is positive. 
\end{proposition}
\begin{proof}
The strength of the $i$-th node $v_i$, $i = 1,\ldots, n$,  is given by $\str{(v_i)} = n k_\square (x)$, for every $x \in   P_i$. We have 
\begin{equation}
k_\square (x) = \int_0^1 \sum_{j=1}^n A_{ij} \chi_{P_j}(y)dy = \frac{1}{n} \sum_{j = 1}^n A_{ij}, \quad \forall x \in P_i, 
\end{equation} 
where $A_{ij}$ was defined by~\cref{eq:quotient-graph}. Hence, 
\begin{align*}
k_\square (x)  = n \int_{P_i} \sum_{j = 1}^n \int_{P_j} W(x',y')dy'dx' 
                     = n \int_{P_i} \int_0^1  W(x',y')dy'dx' 
                     = n \int_{P_i}k(x')dx', 
\end{align*}
showing $k_\square (x) \geq c$
where $c>0$ is the   constant from~\cref{hyp:bounded-away-from-0}. 
\end{proof}
\begin{remark}\label{rmk:solution-on-quotient-graph}
It follows that the finite-dimensional IVP~\cref{eq:IVP-quotient-graph} on the quotient graph has a unique solution given by $e^{t \mathcal L_\square}g_\square$. 
\end{remark}

\begin{theorem} [Convergence with $W/\mathcal P$]
	\label{thm:convergence-quotient-graph} Let $W$ be a connected graphon satisfying \cref{hyp:bounded-away-from-0}, and let   $w(x,t)$ be the solution of  IVP~\cref{eq:IVP}. Further let $u(x,t)$ be the solution of the associated discrete problem~\cref{eq:IVP-quotient-graph}. Then for all $t \in \mathbb R^+$ it holds that 
	\begin{displaymath}
	\norm{u(\cdot,t) - w(\cdot,t)}_2 \rightarrow 0 \quad \mbox{ as } \quad n \rightarrow \infty. 
	\end{displaymath}
\end{theorem}

\begin{proof} Using~\cref{rmk:solution-on-quotient-graph}, by the Minkowski inequality we have 
\begin{align}
\norm{u(\cdot,t)-w(\cdot,t)}_2 & = \norm*{e^{t \mathcal L_\square}g_\square - e^{t \mathcal L}g}_2 \nonumber \\
& = \norm*{\sum_{k = 0}^\infty \frac{t^k}{k!}  \mathcal L_\square^k g_\square - \sum_{k = 0}^\infty \frac{t^k}{k!}  \mathcal L^k g  }_2 \nonumber\\
& \leq \norm{g_\square - g}_2 +\sum_{k = 1}^\infty  \frac{t^k}{k!} \underbrace{ \norm*{   \mathcal L_\square^k g_\square -   \mathcal L^k g  }_2}_{(*)} . \label{eq:u-w-1}
\end{align}
Let us write $\mathcal L^{rw} = \mathcal K - \mathcal I$ where $\mathcal K$ is the operator previously defined in~\cref{eq:compact-op-markov-kernel} and $\mathcal I$ is the identity operator. We have a similar decomposition $\mathcal L_\square^{rw} = \mathcal K_\square - \mathcal I$ for the Laplacian of the step graphon.  For $k \geq 1$ and $0 \leq m \leq k$ let us write $\alpha_{mk} = (-1)^m\binom{k}{m}$, and consider $(*)$ in the right-hand side of~\cref{eq:u-w-1}. Using  Newton's binomial theorem we have
\begin{align}\label{eq:Letak-Lk}
\norm*{   \mathcal L_\square^k g_\square -   \mathcal L^k g  }_2  & 
= \norm*{ (\mathcal K_\square - \mathcal I)^k g_\square -  ( \mathcal K^k-\mathcal I) g}_2 \nonumber \\
&= \norm*{ \sum_{m = 0}^k \alpha_{mk}\mathcal K_\square^{k-m}  g_\square - \sum_{m = 0}^k \mathcal K^{k-m} g}_2 \nonumber \\
& \leq \norm*{ \sum_{m = 0}^{k-1} \alpha_{mk}\left(\mathcal K_\square^{k-m}  g_\square  - \mathcal K^{k-m} g\right)  }_2 + \norm*{\alpha_{kk} \left(g_\square - g\right)}_2\nonumber \\
\intertext{with $\vert \alpha_{mk} \vert  = \binom{k}{m}$ and using \cref{lemma:step-kernel-on-step-graphon},}
& \leq   \sum_{m = 0}^{k-1} \binom{k}{m}\underbrace{ \norm*{ \left(\mathcal K_\square^{k-m}   - \mathcal K^{k-m} \right) g  }_2}_{(**)} + \norm*{ \left(g_\square - g\right)}_2. 
\end{align}
By~\cref{hyp:bounded-away-from-0}  and \cref{prop:deg-of-connected-graphon} there exists some constant $c >0$ such that $k_\square(y) \geq c$ for all $y \in [0,1]$. Further, $0 \leq W \leq 1$ on $[0,1]^2$, and so
\begin{equation}
\norm{K_\square}_2  = \norm*{\frac{\eta(W/\mathcal P)}{k}}_2 = \norm*{\eta (W/\mathcal P)}_2  \norm*{ \frac 1 k_\square}_2  \leq \frac 1 c =: \beta_\square,  
\end{equation}
where $K_\square$ denotes the integral kernel of $\mathcal K_\square$. 
For $\ell \in \mathbb{N}_0$, let us define $\mathcal E_\ell := \mathcal K_\square^{\ell}   - \mathcal K^{\ell} $ and $E_\ell :=   K_\square^{\ell}   -   K^{\ell}$.   Then, applying \cref{lemma:difference-HS-operators} successively $\ell-1$ times to $(**)$ in~\cref{eq:Letak-Lk} with $\ell = k-m$, we obtain
\begin{align}\label{eq:double-star-after-applying-lemma}
\norm{\mathcal E_\ell g}_2 
& \leq \beta_\square^{\ell-1} \norm{E_1}_2 \norm{g}_2 + \norm{\mathcal E_{\ell-1} \mathcal K g}_2 \nonumber\\
& \leq \left(\beta_\square^{\ell-1} +\beta_\square^{\ell-2} \right)\norm{E_1}_2 \norm{g}_2  + \norm{\mathcal E_{\ell-2} \mathcal K^2 g}_2 \nonumber\\
& \, \, \vdots \nonumber \\
&  \leq \left( \sum_{j=1}^{\ell -1}\beta_\square^{\ell -j}  \right)\norm{E_1}_2 \norm{g}_2 +  \norm{\mathcal E_{1} \mathcal K^{\ell  -1} g}_2, \nonumber\\
\intertext{and since $E_\ell$ is the kernel of $\mathcal E_\ell$ if $\ell = 1$, }
&  \leq \left( \sum_{j=1}^{\ell -1}\beta_\square^{\ell -j}  \right)\norm{E_1}_2 \norm{g}_2 +  \norm{ E_{1}}_2   \norm{ \mathcal K^{\ell  -1} g}_2 \nonumber\\
\intertext{and with $\norm{\mathcal K^{\ell -1}g}_2 \leq \norm{K^{\circ (\ell -1)}}_2 \norm{g}_2 \leq \norm{K}_2 ^{\ell - 1}  \norm{g}_2$,}
&  \leq \left(\norm{K}_2^{\ell-1} + \sum_{j=1}^{\ell -1}\beta_\square^{\ell -j}  \right)\norm{E_1}_2 \norm{g}_2 \nonumber\\
& \leq \ell \beta^{\ell-1}\norm{E_1}_2 \norm{g}_2, 
\end{align}
where the last inequality stems from $\beta: = \max \left\{  \norm{K}_2,\beta_\square \right\} \geq 1$. Combining \cref{eq:Letak-Lk} and \cref{eq:double-star-after-applying-lemma} yields
\begin{align} \label{eq:Letak-Lk-2}
\norm*{   \mathcal L_\square^k g_\square -   \mathcal L^k g  }_2  & \leq
\sum_{m = 0}^{k-1} \binom{k}{m}   (k-m) \beta^{k-m-1}\norm{K_\square - K}_2 \norm{g}_2 + \norm*{ \left(g_\square - g\right)}_2 \nonumber \\
& \leq
 \beta^{k-1}\sum_{m = 0}^{k-1} \binom{k}{m}   (k-m)\norm{K_\square - K}_2 \norm{g}_2 + \norm*{ \left(g_\square - g\right)}_2. 
\end{align}
From~\cref{eq:u-w-1,eq:Letak-Lk-2} we obtain
\begin{align}\label{eq:u-w-2}
\norm{u(\cdot,t)-w(\cdot,t)}_2 & \leq  \norm{g_\square - g}_2 +\sum_{k = 1}^\infty  \frac{t^k}{k!} \norm*{ \left(g_\square - g\right)}_2 \nonumber \\ 
&   + \norm{K_\square - K}_2 \norm{g}_2 \sum_{k = 1}^\infty  \frac{t^k}{k!} \beta^{k-1}\sum_{m = 0}^{k-1} \binom{k}{m}   (k-m) \nonumber \\
& = \norm{g_\square - g}_2 e^{t} +\norm{K_\square - K}_2 \norm{g}_2  \sum_{k = 1}^\infty  \frac{t^k}{k!} \beta^{k-1}\sum_{m = 0}^{k-1} \binom{k}{m}   (k-m)\nonumber \\
\intertext{and with $\sum_{m = 0}^{k-1} \binom{k}{m}   (k-m) = k2^{k-1}$,}
& = \norm{g_\square - g}_2 e^{t} +\norm{K_\square - K}_2 \norm{g}_2 \sum_{k = 1}^\infty  \frac{t^k}{k!}k (2\beta)^{k-1} \nonumber \\
& \leq \norm{g_\square - g}_2 e^{t} +\norm{K_\square - K}_2 \norm{g}_2   t\sum_{k = 1}^\infty  \frac{(2\beta t)^{k-1}}{(k-1)!}   \nonumber \\
& \leq \norm{g_\square - g}_2 e^{t} +\underbrace{\norm{K_\square - K}_2}_{(***)} \norm{g}_2 te^{2\beta t}. 
\end{align}
By the Lebesgue differentiation theorem, $g_\square \rightarrow g $ pointwise for almost every $x\in [0,1]$ as  $n\rightarrow \infty$, so that 
\begin{equation}\label{eq:geta-tends-to-g}
\norm{g_\square -g}_2 \xrightarrow[n   \to 0]{} 0
\end{equation}
by dominated convergence~\cite{medvedev2014nonlinearDENSE}. Let us consider $(***)$ in~\cref{eq:u-w-2} : 
\begin{align}\label{eq:Keta-K}
\norm{K_\square - K}_2^2 &  =   \int_{[0,1]^2} \left(    \frac{\eta(W/\mathcal P)(x,y)}{k_\square(x,y)}   - \frac{W(x,y)}{k(y)} \right)^2  dx dy  \nonumber \\
& \leq \esssup\limits_{y \in [0,1]} \frac{1}{k_\square^2(y) k^2(y)} \int_{[0,1]^2} \big(        \eta(W/\mathcal P)(x,y) k(y) - W(x,y) k_\square(y)         \big)^2 dx dy \nonumber \\
& \leq \beta^2\int_{[0,1]^2} \big(        \eta(W/\mathcal P)(x,y) (k(y)-k_\square(y)) \big)^2 dx dy \nonumber \\
& + \beta^2 \int_{[0,1]^2}   \big( (W(x,y) - \eta(W/\mathcal P)(x,y)) k_\square(y)         \big)^2 dx dy \nonumber \\
\intertext{and because $\norm{\eta(W / \mathcal P)}_2 \leq 1$ and $\norm{k_\square}_2 \leq 1$, }
& \leq \beta^2 \left(       \norm{k-k_\square}_2^2     + \norm{W- \eta(W/\mathcal P)}_2^2           \right). 
\end{align}
By the Cauchy-Schwarz inequality, 
\begin{align*}
  \norm{k-k_\square}_2^2  
  & = \int_{0}^{1}    \left(          \int_{0}^{1}  \left(   W(y,z) -  \eta(W/\mathcal P)   (y,z)     \right)  dz  \right)^2  dy  \\
  & \leq  \int_{0}^{1}              \int_{0}^{1} \left(   W(y,z) -  \eta(W/\mathcal P) (y,z)    \right)^2  dz    dy \\
  & = \norm{W-\eta(W/\mathcal P)}_2^2, 
\end{align*}
which together with \cref{eq:Keta-K} yields 
\begin{equation} \label{eq:Keta-K-2}
\norm{K_\square - K}_2^2 \leq 2 \beta^2  \norm{W- \eta(W/\mathcal P)}_2^2. 
\end{equation}
By the same argument leading to \cref{eq:geta-tends-to-g}, we have $\norm{W- \eta(W/\mathcal P)}_2 \rightarrow 0$ as $n\rightarrow \infty$
which with~\cref{eq:Keta-K-2} implies
\begin{equation}\label{eq:Keta-K-3}
\norm{K_\square -K}_2 \xrightarrow[n   \to 0]{} 0. 
\end{equation}
Combining \cref{eq:u-w-2,eq:geta-tends-to-g,eq:Keta-K-3} allows to conclude. 
\end{proof}

\subsection{Convergence on the sampled graph $W_{[n]}$}
\label{sec:convergence-on-sampled-graph}
The case of the discrete problem on $W_{[n]}$ can be handled similarly as the discrete problem on $W / \mathcal P$, and the convergence theorem follows mainly from the observation in~\cref{sec:from-graphons-to-graphs-and-back} that $W_{[n]} \rightarrow W$ at every point of continuity of $W$. The necessary convergence in $L^2$ will follow from the supplemental assumption that the graphon is almost everywhere continuous.  The discrete problem (in its step function form) associated to \cref{eq:IVP} on the sampled graph $W_{[n]}$ reads
\begin{subequations}\label{eq:IVP-sampled-graph}
	\begin{align}
	\frac{\partial }{\partial t}u(x,t) &= \mathcal L_{[n]} ^{rw} u(x,t)\\
	u(x,0) &= g_\square(x)  
	\end{align}
\end{subequations}
where the random walk  Laplacian operator on $W_{[n]}$ satisfies
\begin{equation}
\mathcal L_{[n]} ^{rw} f(x)  = \int_0^1 \frac{\eta \left(   W_{[n]}   \right)(x,y)}{k_{[n]} (y)} f(y) dy - f(x), \quad \forall f \in L^2[0,1], 
\end{equation}
and the initial condition is again averaged on each cell of the partition as in~\cref{eq:initial-cond-on-quotient-graph}. One needs to assume sufficiently large $n$ to guarantee $k_{[n]}$ to be bounded away from~0 and so the Laplacian to be well-defined.
\begin{theorem} [Convergence with $W_{[n]}$] 
	\label{thm:sampled-graph}Let $W$ be a connected, almost everywhere continuous graphon satisfying ~\cref{hyp:bounded-away-from-0} and let $w(x,t)$ be the solution of  IVP~\cref{eq:IVP}. Further let $u(x,t)$ be the solution of the associated discrete problem~\cref{eq:IVP-sampled-graph}. Then for all $t \in \mathbb R^+$ it holds that 
	\begin{displaymath}
	\norm{u(\cdot,t) - w(\cdot,t)}_2 \rightarrow 0 \quad \mbox{ as } \quad n \rightarrow \infty. 
	\end{displaymath}
\end{theorem}
The proof is similar as for~\cref{thm:convergence-quotient-graph}. 
\begin{remark}
	The initial condition could have been sampled in a similar fashion as the graphon, to yield the step function
	$
	g_{[n]}  = \sum_{i = 1}^ng\left(\frac{ i}{n}\right) \chi_{P_i}. 
	$ 
	Almost everywhere continuity of $g$ would ensure that $\norm{g-g_{[n]}}_2 \rightarrow 0$ when $n \rightarrow \infty$, and would be part of the hypothesis of a convergence theorem. The proof of \cref{thm:sampled-graph} would only require minor changes, which are similar to those  discussed next in the new context of~\cref{sec:convergence-on-convergent-graph-sequences}. 
\end{remark}

\subsection{Convergence for  a sequence of discrete problems}
\label{sec:convergence-on-convergent-graph-sequences}
This time we consider a sequence of problems defined on graphs with increasing number of nodes. We assume the sequence of dense connected graphs, say $(G_n)$, converges to a limit graphon $W$ in the $L^2$ metric, in the sense that $\norm{\eta(G_n) - W}_2 \rightarrow 0$ as $n\rightarrow \infty$.  Let   $k_n$ denote the degree function of the empirical graphon $\eta(G_n)$. 
Consider the family of   discrete problems  under the mapping $\eta$
\begin{subequations}\label{eq:IVP-convergent-graph}
	\begin{align}
	\frac{\partial }{\partial t}u(x,t) &= \mathcal L_{n} ^{rw} u(x,t)\\
	u(x,0) &= g_n(x)  \in L^2[0,1],
	\end{align}
\end{subequations}
where the random walk  Laplacian operator $\mathcal L_n^{rw}$ satisfies
\begin{equation}
\mathcal L_n ^{rw} f(x)  = \int_0^1 \frac{\eta \left(  G_n \right)(x,y)}{k_{n} (y)} f(y) dy - f(x), \quad \forall f \in L^2[0,1]. 
\end{equation}
Similarly as before, we write $\mathcal L^{rw}_n  = \mathcal K_n - \mathcal I$. 
\begin{theorem} [Convergence with $(G_n)$] 
	\label{thm:convergence-Gn}Let $(G_n)$ be a sequence of connected graphs that converges to a connected graphon $W$ satisfying \cref{hyp:bounded-away-from-0}. Let   $w(x,t)$ be the solution of the  IVP~\cref{eq:IVP} associated to $W$ with initial condition $w(\cdot, 0) = g \in L^2[0,1]$. Further let $u(x,t)$ be the solution of the corresponding discrete problem~\cref{eq:IVP-convergent-graph}, and assume that $\norm{g_n-g}_2 \rightarrow 0$ as $n \rightarrow \infty$. Then for all $t \in \mathbb R^+$ it holds that 
	\begin{displaymath}
	\norm{u(\cdot,t) - w(\cdot,t)}_2 \rightarrow 0 \quad \mbox{ as } \quad n \rightarrow \infty. 
	\end{displaymath}
\end{theorem}
\begin{proof} 
The proof follows the same steps as for \cref{thm:convergence-quotient-graph}. However, using \cref{lemma:step-kernel-on-step-graphon} to obtain~\cref{eq:Letak-Lk} is now prohibited due to the initial condition of a discrete problem no longer resulting from an  averaging of the continuous IVP. Consider a sufficiently large $n$ such that the degree function of the empirical graphon satisfies $k_n\geq c$ for some constant $c>0$. Not relying this time on~\cref{lemma:step-kernel-on-step-graphon}, we write
\begin{align*}
\norm{\mathcal K_n^{\ell - m} g_n - \mathcal K^{k-m} g}_2  & =  \norm{\mathcal K_n^{\ell - m} g_n - \mathcal K_n^{\ell - m} g + \mathcal K_n^{\ell - m} g - \mathcal K^{k-m} g}_2 \\
& \leq \norm{\mathcal K_n^{\ell - m} (g_n - g)}_2 + \norm{( \mathcal K_n^{\ell - m}- \mathcal K^{\ell - m} )g }_2, 
\end{align*}
with the first term in the right-hand side newly present. Following the same steps leading to~\cref{eq:Letak-Lk}, we obtain 
\begin{multline*}
\norm*{   \mathcal L_n^k g_n -   \mathcal L^k g  }_2  
\leq    \sum_{m = 0}^{k-1} \binom{k}{m} \norm{\mathcal K_n^{\ell - m} (g_n - g)}_2  \\
 + \sum_{m = 0}^{k-1} \binom{k}{m}\norm*{ \left(\mathcal K_n^{k-m}   - \mathcal K^{k-m} \right) g  }_2  + \norm*{ \left(g_n - g\right)}_2, 
\end{multline*}
where again the first term right of the inequality is new. In fashion similar  to the proof of \cref{thm:convergence-quotient-graph}, with $\beta := \max \left\{   \norm{K}_2 , \norm*{\frac{1}{k_{n}}}_\infty     \right\}$ we have
\begin{multline*} 
\norm*{  u(\cdot,t) - w(\cdot,t)  }_2   \leq     \sum_{k = 1}^\infty  \frac{t^k}{k!} \sum_{m = 0}^{k-1} \binom{k}{m}   \norm{\mathcal K_n^{\ell - m} (g_n - g)}_2 \\
+ \norm{g_n - g}_2 e^{t} +\norm{K_n - K}_2  \norm{g}_2 te^{2\beta t}. 
\end{multline*}
Using $\norm{\mathcal K^{k-m}} \leq \beta^{k-m} \leq \beta^k$ and $\sum_{m = 0}^{k-1} \binom{k}{m}  \leq 2^k$, we have
\begin{equation*}
\sum_{k = 1}^\infty  \frac{t^k}{k!} \sum_{m = 0}^{k-1} \binom{k}{m}   \norm{\mathcal K_n^{\ell - m} (g_n - g)}_2
\leq \sum_{k = 1}^\infty  \frac{t^k}{k!} 2^k  \beta^k  \norm{(g_n - g)}_2
= \norm{  (g_n - g)}_2e^{2 \beta t}, 
\end{equation*}
leading to 
\begin{displaymath}
\norm*{  u(\cdot,t) - w(\cdot,t)  }_2   \leq  \norm{g_n - g}_2 \left(e^{t} + e^{2 \beta t}\right) +\norm{K_n - K}_2  \norm{g}_2 te^{2\beta t}. 
\end{displaymath}
\end{proof}

\section{Relaxation}
\label{sec:relaxation}
The evolution of a system towards its asymptotic state $w_\infty$ starting from any initial condition is know as relaxation. The so-called relaxation time characterizes the rate of this evolution.  In the continuum limit of the node-centric walk, it is determined by the spectral properties of $\mathcal{K}$, in a way reminiscent of random walks on  finite graphs. For the node-centric continuous-time walk, we will show now that this rate can be exponential.  Let us define a normalized adjacency operator, which is then used in the definition of a normalized Laplacian. 
\begin{definition}[Normalized adjacency operator] Under \cref{hyp:bounded-away-from-0}, let the normalized adjacency operator be the 
	integral operator $\mathcal A^{norm} : L^2[0,1] \rightarrow L^2[0,1]$ defined by $\mathcal A^{norm} = \mathcal M_{1/\sqrt k} \, \mathcal K \mathcal M_{\sqrt{k}}$. 
\end{definition} 
Observe that under \cref{hyp:bounded-away-from-0} the kernel $ W(x,y) / \sqrt{k(x)k(y)} $ 
of $\mathcal{A}^{norm}$ is square-integrable and symmetric. Hence   $\mathcal A^{norm}$ is a compact, self-adjoint Hilbert-Schmidt integral operator and the Hilbert-Schmidt theorem 
applies. 
Therefore, there exists an orthonormal basis of eigenfunctions $\left\{\phi_m \right\}$ with associated eigenvalues $\theta_m$, so that operator $\mathcal A^{norm}$ has the canonical form 
\begin{equation} \label{eq:S-canonical}
\mathcal A^{norm} = \sum_{m=1}^\infty \theta_m \left(\phi_m, \cdot\right) \phi_m. 
\end{equation}
The operator $\mathcal{L}^{norm} : = \mathcal A^{norm}  - \mathcal I$ is the associated normalized (or sometimes also called symmetric) Laplacian. 
Note that for $\ell \in \mathbb N$,   $\left(\mathcal A ^{norm}\right)^\ell$ has eigenfunctions $\phi_m$ and eigenvalues $\theta_m^\ell$, and that $\left(  \mathcal L^{rw}   \right)^\ell =  \mathcal M_{\sqrt{k}}   \left( \mathcal L ^{norm} \right)^\ell \mathcal  M_{1/\sqrt k} $. Combined with~\cref{eq:S-canonical} this yields the singular value decomposition
\begin{align} \label{eq:spectral-form-semigroup}
	e^{\mathcal L^{rw} t} 
	& =\mathcal M_{\sqrt{k}} \left(   \sum_{\ell = 0}^\infty  \frac{t^\ell}{\ell!}            \left(      \sum_{m=1}^\infty \theta_m \left( \phi_m,\cdot  \right) \phi_m - \mathcal I   \right)^\ell          \right)   \mathcal M_{1/\sqrt k} \nonumber \\
	&=\mathcal M_{\sqrt{k}} \left(   \sum_{\ell = 0}^\infty  \frac{t^\ell}{\ell!}            \left(      \sum_{m=1}^\infty (\theta_m-1) \left( \phi_m,\cdot  \right) \phi_m  \right)^\ell          \right)   \mathcal M_{1/\sqrt k} \nonumber \\
	&=\mathcal M_{\sqrt{k}} \left(    \sum_{m=1}^\infty  \sum_{\ell = 0}^\infty  \frac{t^\ell}{\ell!}    \lambda_m ^\ell         \left(     \phi_m,\cdot  \right) \phi_m          \right)   \mathcal M_{1/\sqrt k} \nonumber \\
	&=   \sum_{m=1}^\infty e^{\lambda_m  t}         \left(      \frac{\phi_m}{\sqrt{k}},\cdot  \right) \sqrt{k}\phi_m    
\end{align}
with  $\lambda_m = \theta_m - 1$ the eigenvalues of $\mathcal L^{norm}$. By letting $\psi_m  = \frac{\phi_m}{\sqrt{k}}$ and $\zeta_m  = \sqrt{k}\phi_m  $, the solution of IVP~\cref{eq:IVP} reads
\begin{equation}\label{eq:IVP-solution-spectral-form}
w(x,t) =  \sum_{m=1}^\infty e^{\lambda_m  t}         \left(   \psi_m     , g   \right) \zeta_m(x). 
\end{equation}
The following proposition allows for a characterization of  the rate of the relaxation towards $w_\infty$. 
\begin{proposition}
	\label{prop:eigenvalues}
	Let $W$ be a graphon satisfying  \cref{hyp:bounded-away-from-0}, then the eigenvalues $\lambda_m$ of $\mathcal L^{norm} $ are non-positive reals, and the largest eigenvalue is zero. If moreover $W$ is connected, then the eigenvalue zero has multiplicity one. 
\end{proposition}
\begin{proof}
	That the eigenvalues are reals results from  $\mathcal L^{norm} $ being  a self-adjoint operator on $L^2[0,1]$. Let $\lambda$ be an eigenvalue of $\mathcal L^{norm}$ associated to $\phi$. Then $\lambda$ is given by the Rayleigh quotient
	\begin{equation} \label{eq:Rayleigh-quotient}
	\lambda
	= \frac{\left(      \lambda \phi, \phi\right)}{ \left(\phi, \phi \right)  } 
	= \frac{\left(      \mathcal L^{norm}  \phi, \phi\right)}{ \left(\phi, \phi \right)  }. 
	\end{equation}
	Consider the numerator in the right-hand side  of \cref{eq:Rayleigh-quotient}. For all $f \in L^2[0,1]$ we can write
	\begin{align}
		\left(    \mathcal L^{norm}  f, f\right) 
		&  = \int_0^1 \int_0^1 \frac{W(x,y)}{\sqrt{k(x)} \sqrt{k(y)}} f(x)f(y) dxdy - \int_0^1 f^2(x) dx\nonumber \\
		& = \frac{1}{2} \Bigg(
		2 \int_0^1 \int_0^1  \frac{\sqrt{W(x,y)}}{\sqrt{k(x)} }   \frac{\sqrt{W(x,y)}}{\sqrt{k(y)} }    f(x)f(y) dxdy \nonumber \\
		& \hphantom{=} - \int_0^1 \int_0^1 \frac{W(x,y)}{k(x)}  f^2(x)   dx dy
		- \int_0^1 \int_0^1 \frac{W(x,y)}{k(y)}  f^2(y)   dx dy \Bigg) \nonumber \\
		& = -\frac{1}{2} \int_0^1 \int_0^1 \left(            \frac{\sqrt{W(x,y)}}{\sqrt{k(x)} } f(x) -  \frac{\sqrt{W(x,y)}}{\sqrt{k(y)} } f(y)             \right)^2dxdy,  \label{eq:rayleigh-num-1}
	\end{align}
	which is non-positive. The   claim that zero is an eigenvalue follows from the fact  that $\mathcal L^{norm}   \sqrt{k(x)} = 0$  on $[0,1]$. Finally, let us show that the nullspace of $\mathcal L^{norm} $ has dimension one if $W$ is connected. By defining $g = \mathcal M_{1/\sqrt{k}  }f$ on $[0,1]$, we have to show that if the right-hand side of  \cref{eq:rayleigh-num-1} is zero, namely
	\begin{equation}
	\label{eq:RHS}
	-\frac{1}{2} \iint_{[0,1]^2}W(x,y)  \left(         g(x) - g(y) \right)^2dxdy = 0, 
	\end{equation}
	then $g$ has to be a constant function on $[0,1]$. By contradiction, assume that there exists some non-constant function $g$ that verifies \cref{eq:RHS}. For simplicity, consider the case that $g = c_1$ on some $S \in \mathfrak M [0,1]$ with $\mu (S) \in (0,1)$, and $g = c_2 $ on $ S^c : = [0,1] \setminus S$, with $c_1, c_2 \in \mathbb{R}$, $c_1 \neq c_2$. The reasoning would be similar if $g$ is a piecewise constant function on any other partition of $[0,1]$, and can be extended by density to any  not piecewise constant $g$.  Based on \cref{eq:RHS}, we can write 
	\begin{equation}
	0  =  \iint_{[0,1]^2}W(x,y)  \left(         g(x) - g(y) \right)^2dxdy \geq  \iint_{S \times S^c}W(x,y)  \left(         g(x) - g(y) \right)^2dxdy, 
	\end{equation}
	and the integral in the right-hand side is zero. Since  $W$ is connected,  we have
	$
	\iint_{S \times S^c}W(x,y) dxdy >0
	$, 
	and hence there exists a positive-measured subset $E \times F $ of $S \times S^c$ such that $W >0$ on $E \times F$. Therefore, $  g(x) - g(y)  = 0$ for almost every $(x,y) \in E \times F$. But then, since $g = c_1$ on $E \subset S$, $g = c_1$ on $F \subset S^c$, a contradiction. 
\end{proof}
\begin{remark}[Spectral gap]\label{rmk:spec-gap-Lnorm}
	The last claim of~\cref{prop:eigenvalues} means that  the spectral gap of $\mathcal L^{norm}$, namely the positive difference between the largest  and the second largest eigenvalue, is nonzero when the graphon is connected with degree function bounded away from zero. Observe that if  $k$ is not bounded away from zero, 
	we may no longer write $\mathcal A^{norm} = \mathcal M_{1/\sqrt k} \, \mathcal K \mathcal M_{\sqrt{k}}$ because $1/ \sqrt k$ is not bounded. This implies that the spectrum of $\mathcal L^{rw}$ can no longer be deduced directly from the spectrum of the compact self-adjoint operator $\mathcal L^{norm}$. However, the eigenvalues of $\mathcal L^{rw}$ may in some cases be computed directly, see \cref{ex:separable-graphon} hereafter. 
	If the graphon is not connected, one can analyze the dynamics on each connected component independently, as follows from the decomposition introduced in~\cite{janson2008connectedness}. Therefore, it only remains to consider relaxation in the case of a connected graphon where \cref{hyp:bounded-away-from-0} is not satisfied, meaning $k(x)$ becomes arbitrarily small on positive measured subsets of $[0,1]$, and such that $\mathcal K$ is still well-defined. In this situation the eigenvalues of  the generally non-self-adjoint operator $\mathcal L^{rw}$ may be embedded in the continuous spectrum whilst in the discrete or discretized version, the spectrum is composed only of eigenvalues and there can be a positive spectral gap as a result of connectedness in a finite graph. 
\end{remark}

\begin{example}[Eigenvalues of $\mathcal L^{rw}$ on a separable graphon] \label{ex:separable-graphon}
	Consider the separable\footnote{The   graphon $W(x,y)$ is separable if it can be written as   $W(x,y) = \zeta(x) \zeta(y)$ for some function~$\zeta$. } graphon $W(x,y) = xy$. The degree function is $k(x) = x/2$, in which case  $\mathcal L^{rw} = \mathcal K - \mathcal I $ with $\mathcal Kf(x) = 2x \int_0^1 f(y)dy$.  Any eigenvalue $\lambda_{\mathcal K}$ of $\mathcal K$ satisfies 
	\begin{equation}\label{eq:eigenpairs-K}
	2x \int_0^1 \phi_{\mathcal K}(y) dy = \lambda_{\mathcal K}\phi_{\mathcal K} (x), \quad x \in [0,1], 
	\end{equation}
	where $\phi_{\mathcal K}$ is an eigenfunction. It suffices to subtract one to the eigenvalues of $\mathcal K$ and to hold the same eigenfunctions to obtain the eigenpairs of $\mathcal L^{rw}$. From~\eqref{eq:eigenpairs-K}, one finds  $\lambda_{\mathcal K} = 1$ with the one-dimensional eigenspace $\spanset \lbrace x \rbrace$, or $\lambda_{\mathcal K} = 0$ with the infinite-dimensional eigenspace $\lbrace1\rbrace^\perp$. Observe these spaces are not orthogonal, but their sum is $L^2[0,1]$. 
\end{example}

\section{Extension to the discrete-time walk}
\label{sec:discrete-time}
The analysis of the node-centric continuous-time walk carries over to the discrete-time version~\cref{eq:DTRW}. 
The corresponding  IVP on the continuum reads
\begin{subequations}\label{eq:DTRW-IVP-discrete}
	\begin{align}
	w(\cdot, \ell+1) & = \mathcal K w(\cdot,\ell),  \quad  \ell  \in \mathbb N_0 \\
	w(\cdot,0) & = w_0 \in L^2[0,1], 
	\end{align}
\end{subequations}
with solution given by
$
w(\cdot,\ell ) = \mathcal K^\ell w_0
$
for every $\ell \in \mathbb{N}$. 

Following the same steps as in~\cref{sec:continuum-limit-node-centric,sec:well-posedness,sec:Convergence-on-dense-weighted-graphs}, we obtain a similar convergence result  on the quotient graph $W/\mathcal P$. Indeed, the proof  in continuous time already contains the necessary bound on the operator norm of the difference  $\mathcal K^\ell\mathcal -  \mathcal K_\square^\ell$, see \cref{eq:double-star-after-applying-lemma}. The same holds true on the sampled graph $W_{[n]}$ and for a sequence of discrete problems. 
Analogously as for~\cref{eq:IVP-solution-spectral-form},  the spectral expansion of the  solution of the discrete-time IVP~\cref{eq:DTRW-IVP-discrete} reads
\begin{equation} \label{eq:relax-time-disc}
w(\cdot,\ell ) = \sum_{m=1}^\infty   \lambda_m^\ell  \left(   \psi_m, w_0 \right) \zeta_m , \quad \ell \geq 0. 
\end{equation}

\section{Conclusion}
\label{sec:conclusions}

There are two main arguments motivating this work. On the one hand, random walks and Laplacians play a central role in the study of graphs, and a better understanding of their behavior on graphons has a clear mathematical interest, with theoretical and algorithmic objectives. On the other hand, even though large networks become ever  more common in numerous fields of research, a rigorous study of the continuum limit of the different types of random walks on graphs was still lacking. 

This paper is intended as a first step towards a systematic study of classes  of random walks on discrete domains, relying on the adequate framework provided by graph-limit theory. We have first shown that the continuum-limit of the discrete heat equation~\cite{medvedev2014nonlinearDENSE} could be interpreted as the limit of a rescaled edge-centric continuous-time Poisson random walk. We have then studied the continuum limit of the remaining two fundamental classes of random walks on graphs, which complement the discrete heat equation : the discrete-time walk, and its continuous-time generalization. A final part of the document was devoted to spectral aspects of the introduced random walk Laplacian operator, thereby allowing for a characterization of the relaxation time of the process. 

The world of random walks is a very broad one, and   this initial work   calls for extensions. A  promising research direction would consist in generalizing the  semigroup approach developed here, or the one  in~\cite{medvedev2014nonlinearDENSE,medvedev2014small,medvedev2018kuramoto},  to the diverse classes of random walk processes omitted here, for instance  walks on temporal or directed networks. 
A second line of research could focus on the case of sparse graphs, and the way they affect the approximation procedure we have applied.  Sparsity is indeed known to be the norm rather than the exception in real-life networks.  Such extension was  already provided for the graph-limit version of the heat equation, using $L^p$ graphons~\cite{Borgs2014Lptheory,medvedev2014small,kaliuzhnyi2017semilinear}. Another possible venue of investigation could follow from recent work on sparse exchangeable graphs generated via graphon processes or graphexes~\cite{borgs2011limitsrandomsequences,  borgs2016sparse,caron2017sparse}. 

\bibliographystyle{siamplain}
\bibliography{RWgraphons_refs}
\end{document}